\documentclass[11pt,reqno]{amsart}
\usepackage{amsmath, amssymb, graphicx}
\usepackage{lscape}
\usepackage[colorlinks=true]{hyperref}

\newtheorem{thm}{Theorem}

\newtheorem{remark}{Remark}

\begin{document}

\title{Series for $1/\pi$ Using Legendre's Relation}

\author{J. G. Wan}
\address{School of Mathematical and Physical Sciences, The University of Newcastle, Callaghan, NSW 2308, Australia}
\email{james.wan@newcastle.edu.au}

\date{\today}

\maketitle

\begin{abstract}
We present a new method for producing series for $1/\pi$ and other constants using Legendre's relation, starting from a generation function that can be factorised into two elliptic $K$'s; this way we avoid much of  modular theory or creative telescoping. Many of our series involve special values of Legendre polynomials; their relationship to the more traditional Ramanujan series is discussed.
\end{abstract}

\section{Introduction}

Ramanujan-type series for $1/\pi$ have been extensively studied since \cite{Ra}. Originally taking the form 
\begin{equation} \label{rama}
\sum_{n=0}^\infty\frac{(\frac12)_n(s)_n(1-s)_n}{n!^3}(a+bn)z_0^n =\frac c\pi,
\end{equation}
where $s \in \{1/2,1/3,1/4,1/6\}$, many such series were found to be rational -- that is, they enjoy the property $a,b,z_0,c^2 \in \mathbb{Q}$. Those rational series are of theoretical as well as practical interest. Recently, with works such as \cite{CTYZ} and \cite{WZ}, more general series are being studied. A more encompassing series for $1/\pi$ would take the form
\begin{equation} \label{genseries}
\sum_{n=0}^\infty U(n) \,p(n)\,z_0^n =\frac{c}{\pi^k},
\end{equation}
where $U(n)$ is an arithmetic sequence, and $p(n)$ is a polynomial (often linear or quadratic in $n$). For instance, in \cite{WZ}, $U(n)$ can take the form of an Ap\'ery-like sequence times a Legendre polynomial evaluated at a special argument, $P_n(x_0)$.

Most of the current methods for producing such series rely on one of the following methods: 
\begin{itemize}
\item  Hypergeometric series (through the use of Clausen's formula and singular values of the complete elliptic integral $K$), pioneered by the Borweins \cite{Bor};
\item Modular machinery (taking advantages of algebraic relations of modular forms of the same weight, e.g.~modular equations) together with the first approach, see for example \cite{Chud}, and more recently \cite{CTYZ, CWZ};
\item Experimental mathematics (creative telescoping with the Wilf-Zeilberger algorithm), explored in say \cite{JG2};
\item Summation formulas for hypergeometric series at special arguments, Fourier-Legendre series, and other miscellaneous methods, for instance see \cite{chu} and \cite{glaisher}.
\end{itemize}
There exists a huge literature apart from the ones we cited above; we cannot even give a brief account for them here since our approach is quite different, so we only direct interested readers to the survey article \cite{survey}.

A notable feature of $1/\pi$ series produced using the methods above has been severe restrictions in the argument of the geometric term ($z_0$ in \eqref{genseries}), as $z_0$ may need to come from singular values of $K$, or be a special value for a summation formula to work. Here we give a new method for producing series for $1/\pi$, as well as some related constants, using only Legendre's relation. The method presented here breaks such restrictions so the argument can be any real number for which the underlying series converges.

\section{Legendre's relation}

Our analysis hinges on Legendre's relation \cite[Theorem 1.6]{Bor}, which states 
\begin{equation} \label{legendrer}
E(x)K'(x)+E'(x)K(x)-K(x)K'(x) = \frac{\pi}{2}.
\end{equation}
As usual, in hypergeometric notation 
\[K(x) = \frac{\pi}{2}{_2F_1}\biggl({{\frac12,\frac12}\atop 1};x^2\biggr)\] is the complete elliptic integral of the first kind, $E(x)$ is the complete elliptic integral of the second kind, and $K'(x)$ denotes $K(x')$, where $x' = \sqrt{1-x^2}$ is the complementary modulus. Equation \eqref{legendrer} can be easily proven by differentiating both sides and showing that they agree at one point (say at $x=1/\sqrt{2}$). A more general form of \eqref{legendrer} in fact holds \cite[equation (5.5.6)]{Bor}:
\begin{equation} \label{legendregen}
E_s(x)K_s'(x)+E_s'(x)K_s(x)-K_s(x)K_s'(x) = \frac{\pi}{2}\frac{\cos(\pi s)}{1+2s},
\end{equation}
where 
\[ K_s(x) = \frac{\pi}{2}\,{_2F_1}\biggl({{\frac12-s,\frac12+s}\atop 1};x^2\biggr), \quad E_s(x) = \frac{\pi}{2}\,{_2F_1}\biggl({{-\frac12-s,\frac12+s}\atop 1};x^2\biggr). \]
Set $s=0$ in \eqref{legendregen} and we recover \eqref{legendrer}. Note also that in \eqref{legendregen}, $s$ is not restricted to the four values as in \eqref{rama}. \\

Suppose we have a factorisation of the following type:
\begin{equation} \label{key}
 \pi^2 G(z) = K(a(z)) K(b(z)),
\end{equation}
where $G$ is analytic near the origin and satisfies an ordinary differential equation of degree no less than 4 -- for instance, $G$ could be a $_4F_3$. (The condition on the degree of the differential equation for $G$ is imposed because we will solve a system of four equations below, so having three linearly independent derivatives help.) Suppose further that we can find a number $z_0$ such that $a(z_0)^2 = 1-b(z_0)^2$, so that the right hand side of \eqref{key} becomes $K(a(z_0))K'(a(z_0))$. We then consider a linear combination of derivatives of equation \eqref{key}, namely
\begin{align} 
& \pi^2\bigl(A_0 G(z_0) + A_1 \frac{\mathrm d}{\mathrm d z}G(z_0) + A_2 \frac{\mathrm d}{\mathrm d z^2}G(z_0) + A_3 \frac{\mathrm d}{\mathrm d z^3}G(z_0)\bigr) \nonumber \\
= & B_0 KK'(z_0) + B_1 EK'(z_0)+ B_2 E'K(z_0) + B_3 EE'(z_0), \label{diff3}
\end{align}
where $A_i$ are constants that may depend on $z_0$, while $B_i$ depend on $A_i$. The equality in \eqref{diff3} holds because derivatives of $E$ and $K$ are again expressible in terms of $E$ and $K$.
It remains to solve (if possible) the following system of equations for $A_i$,
\[ B_0=-1, \ B_1 = 1, \ B_2 = 1, \ B_3 = 0, \]
so that we may apply Legendre's relation \eqref{legendrer} to \eqref{diff3} and obtain, for those choices of $A_i$,
\begin{equation} \label{solved} A_0 G(z_0) + A_1 \frac{\mathrm d}{\mathrm d z}G(z_0) + A_2 \frac{\mathrm d}{\mathrm d z^2}G(z_0) + A_3 \frac{\mathrm d}{\mathrm d z^3}G(z_0) = \frac{1}{2\pi}. \end{equation}
A series for $1/\pi$ is thus obtained; when written as a sum, the left hand side typically contains a cubic of the summation variable. We will illustrate such series using different choices of $G$ below.

\section{Brafman's formula}

An example of a factorisation in the form of \eqref{key} comes from Brafman's formula \cite{Br1} involving the Legendre polynomials,
\[ P_n(x)={}_2F_1\biggl({{-n,  n+1}\atop 1}; \frac{1-x}2 \biggr). \]
Brafman's formula has been used to produce (a different type of) series for $1/\pi$ in \cite{CWZ, WZ}; it states that
\begin{equation} \label{braf1}
\sum_{n=0}^\infty\frac{(s)_n(1-s)_n}{n!^2}P_n(x)z^n
={}_2F_1\biggl({{s,1-s}\atop 1};\alpha\biggr)\, {}_2F_1\biggl({{s,1-s}\atop 1};\beta\biggr),
\end{equation}
where $\alpha=(1-\rho-z)/2$, $\beta=(1-\rho+z)/2$, and $\rho=(1-2xz+z^2)^{1/2}$.

Although equation \eqref{braf1} is of type \eqref{key}, solving for $\alpha^2=1-\beta^2$ only results in a trivial identity. Therefore our strategy is to modify the arguments $\alpha$ or $\beta$ via some transformations.

\medskip

\subsection{The $s=1/2$ case}

Using $s=1/2$ and applying a quadratic transformation of $K$ \cite[theorem 1.2]{Bor} to one of the terms in \eqref{braf1}, we obtain
\begin{equation} \label{braf2} \frac{\pi^2}{4} \sum_{n=0}^\infty \frac{(\frac12)_n^2}{n!^2} P_n(x)z^n = \frac{1}{1+\alpha^{1/2}} \, K\biggl(\frac{2\alpha^{1/4}}{1+\alpha^{1/2}}\biggr) K\bigl(\beta^{1/2}\bigr). \end{equation}

This fits the type of \eqref{key}. After significant amount of algebra as outlined by the approaches leading to \eqref{solved}, we have the following: 
\begin{thm} \label{thm1leg} For $k \in (0,1)$,
\small
\begin{align} \nonumber & \sum_{n=0}^\infty \binom{2n}{n}^2 P_n\biggl(\frac{-k^4 + 6 k^3 - 2k + 1}{(k^2 + 1) (k^2 + 2 k - 1)}\biggr) \biggl(\frac{(k^2 + 1) (k^2 + 2 k - 1)}{16(k + 1)^2}\biggr)^n  \bigl(C_3 n^3 + C_2 n^2+C_1 n + C_0\bigr)  \\ 
&  = \frac{2(k+1)^3(k^2+1)}{\pi}, \label{thm1state}
\end{align} 
where 
\begin{align*}
C_3 & = 4 (k-1)^2 k^2 (k^2 + 3 k + 4)^2, \\
C_2 & = 12 (k - 1) k (k^6 + 5 k^5 + 10 k^4 + 10 k^3 + 5 k^2 - 3 k + 4), \\
C_1 & = 9 k^8 + 36 k^7 + 37 k^6 + 8 k^5 - 9 k^4 - 56 k^3 + 63 k^2 - 28 k + 4, \\
C_0 & = (k^2 + 2k - 1)^2 (2 k^4 + 3 k^2 - 2 k + 1).
\end{align*} \normalsize
\end{thm}

\begin{proof}
A little algebra shows that if we choose
\[ x = \frac{1-2k+6k^3-k^4}{(k^2+1)(k^2+2k-1)}, \ z_0 = \frac{(k^2+1)(k^2+2k-1)}{(k+1)^2}, \]
then, viewing $\alpha$ and $\beta$ as functions of $z$, we get $\beta(z_0)^{1/2} = k$, and $2\alpha(z_0)^{1/4}/(1+\alpha(z_0)^{1/2}) = \sqrt{1-k^2}$, as desired. With these choices we have $\alpha(z_0) = (1-k)^2/(1+k)^2$; we can also compute and simplify the derivatives $a'(z), a''(z), a'''(z)$ and $b'(z), b''(z), b'''(z)$ at $z=z_0$. Thus, as in \eqref{key}, we have an equation of the type
\[ \pi^2 \biggl[\frac{1+\alpha(z)^{1/2}}{4} \sum_{n=0}^\infty \frac{(\frac12)_n^2}{n!^2} P_n(x)z^n\biggr] = K\biggl(\frac{2\alpha(z)^{1/4}}{1+\alpha(z)^{1/2}}\biggr) K\bigl(\beta(z)^{1/2}\bigr), \]
where at $z=z_0$ the arguments of the two $K$'s are complementary. 

We take a linear combination (with coefficients $A_i$) of the $z$-derivatives of the above equation, as done in \eqref{diff3}, then substitute in $z=z_0$ and simplify the resulting expression using the precomputed values for $\alpha'(z_0), \beta'(z_0)$ etc. Finally, we solve for $A_i$ so that Legendre's relation may be applied to obtain a series of the form \eqref{solved}. The result, after tidying up, is \eqref{thm1state} (where we have replaced the Pochhammer symbols by binomial coefficients).

We now look at the convergence. From the standard asymptotics for the Legendre polynomials, we have, as $n \to \infty$,
\[ P_n(x) =O\bigl( \bigl(|x|+\sqrt{x^2-1}\bigr)^n\bigr) \ \mathrm{for} \ |x|>1 \ \mathrm{and} \ P_n(x) =O\bigl( n^{-1/2}\bigr) \ \mathrm{for} \ |x|\le 1. \]
Therefore, for any rational $k \in (0,1)$, the sum in \eqref{thm1state} converges geometrically, where the rate is given by 
\[ \frac{1-2k+6k^3-k^4}{(1+k)^2}+4\biggl(\frac{k(1-k)}{1+k}\biggr)^{3/2}. \]
Note that this is a convex function in $k$ with minimum at $(\sqrt2-1, 8(\sqrt2-1)^3)$ and maxima at $(0,1)$ and $(1,1)$.
\end{proof}

Note that any rational choice of $k \in (0,1)$ leads to a \textit{rational} series in Theorem \ref{thm1leg}, which is indicative that such series are likely to be fundamentally different from ones that are entirely modular in nature (see e.g.~\cite{CWZ}), whose arguments are much more restricted. For instance, with the choice of $k=1/2$ in Theorem \ref{thm1leg}, we get
\[ \sum_{n=0}^\infty \binom{2n}{n}^2 P_n\biggl(\frac{11}{5}\biggr)\biggl(\frac{5}{576}\biggr)^n(14-171n-4452n^2+2116n^3) = \frac{2160}{\pi}, \]
while with $k=2/3$, we have
\[ \sum_{n=0}^\infty \binom{2n}{n}^2 P_n\biggl(\frac{101}{91}\biggr)\biggl(\frac{91}{3600}\biggr)^n(5537 + 11304 n - 173328n^2 + 53824 n^3) = \frac{87750}{\pi}. \]

Theorem \ref{thm1leg} is by no means the unique consequence of \eqref{braf1} with $s=1/2$. For example, we can apply quadratic transformations to both arguments on the right hand side of \eqref{braf1}. The result is also a rational series, convergent for $k \in (0,1)$ and genuinely different from Theorem \ref{thm1leg}, though the general formula is too messy to be exhibited here. We give only one instance (with the choice $k=1/2$) here:
\[ \sum_{n=0}^\infty \binom{2n}{n}^2 P_n\biggl(\frac{19}{13}\biggr)\biggl(\frac{65}{20736}\biggr)^n(97756868 n^3 - 24254580 n^2 - 539415n - 264590)= \frac{6065280}{\pi}. \]
As another example, if we apply to one term in \eqref{braf1} a cubic transformation (corresponding to the rational parametrisation of the cubic modular equation), 
\begin{equation} \label{cubicm}
K\biggl(\frac{p^{1/2}(2+p)^{3/2}}{(1+2p)^{3/2}}\biggr) =(1+2p)\, K\biggl(\frac{p^{3/2}(2+p)^{1/2}}{(1+2p)^{1/2}}\biggr),
\end{equation}
then after a lot of work it is possible to obtain a general, rational series convergent for $p \in (0,1)$. At $p=1/2$ for instance, we get the series
\[ \sum_{n=0}^\infty \binom{2n}{n}^2 P_n\biggl(\frac{353}{272}\biggr)\biggl(\frac{17}{2^{11}}\biggr)^n (44100 n^3 - 30420 n^2 - 1559n - 206)= \frac{8704}{\pi}. \]
However, it is important to note that not all transformations lead to series of type \eqref{solved}.  \\

We give another general theorem for the $s=1/2$ case here. Recall that one of Euler's hypergeometric transformations leads to
\begin{equation}\label{eulert}
K(x) = \frac{1}{\sqrt{1-x^2}}\, K\biggl(\sqrt{\frac{x^2}{x^2-1}}\biggr).
\end{equation}
If we apply a quadratic transformation to one argument of \eqref{braf1} and Euler's transformation \eqref{eulert} to the other, the result is also rational series with at most a quadratic surd on the right hand side.  Once again convergence is easy to establish (the rate is $|z_0| = (1+k)(4k^2-3k+1)/(4k)$), and the general solution recorded below is proven in exactly the same way as Theorem \ref{thm1leg}.

\begin{thm} \label{thm2leg} For $k \in \bigl(\frac{\sqrt{41}-5}{8},1\bigr)$,
\small
\begin{align}
& \nonumber \sum_{n=0}^\infty \binom{2n}{n}^2 P_n\biggl(\frac{1-3k+2k^2-2k^3}{4k^2-3k+1}\biggr) \biggl(\frac{-(1+k)(4k^2-3k+1)}{64k}\biggr)^n \bigl(C_3 n^3 + C_2 n^2+C_1 n + C_0\bigr) \\
& = \frac{8k^{3/2}(4k^2-3k+1)}{\pi},
\end{align}
where
\begin{align*}
C_3 & = \frac{4(k-1)^2}{k+1}(2k-1)(4k^2+3k+1)^2, \\
C_2 & = 12(k-1)(2k-1)(16k^4+k^2-1), \\
C_1 & = 288k^6-400k^5+102k^4+97k^3-93k^2+47k-9, \\
C_0 & = 2 (32k^6-44k^5+9k^4+16k^3-14k^2+6k-1).
\end{align*}
\normalsize
\end{thm}

Examples include
\[  \sum_{n=0}^\infty \binom{2n}{n}^2 P_n\biggl(\frac13\biggr) \biggl(\frac{-1}{36}\biggr)^n (1-3n-84n^2-121n^3) = \frac{18\sqrt3}{\pi}, \]
from $k=1/3$, and when $k=1/2$ (chosen so that $C_3$ vanishes),
\begin{equation} \label{weird}
 \sum_{n=0}^\infty \binom{2n}{n}^2 P_n\biggl(\frac12\biggr) \biggl(\frac{3}{128}\biggr)^n (3+14n) = \frac{8\sqrt2}{\pi}.
\end{equation}
The formula \eqref{weird} is particularly interesting, because although it fits the form of the $1/\pi$ series considered in \cite{CWZ} perfectly, it cannot be explained by the general theory of \cite{CWZ} (in the notation used there, its $\tau_0$ is $iK(\sqrt3/2)/(2K(1/2))$, which is not a quadratic irrationality).\\

Just as in \cite{CWZ}, we can produce `companion series' using Legendre's relation; one example is
\begin{align*} 
\sum_{n=0}^\infty \binom{2n}{n}^2 \Bigl(\frac{3}{128}\Bigr)^n & \biggl[14n(196n^2+196n-3)P_{n-1}\Bigl(\frac12\Bigr) \\
& -(1372n^3+3024n^2+1631n+375)P_n\Bigl(\frac12\Bigr)\biggr]  = \frac{400\sqrt2}{\pi}.
\end{align*}

\begin{remark} \label{rmk1leg}
One might wonder what happens if we set $a(z)=b(z)$ in \eqref{key}. In the case of \eqref{braf2}, as the quadratic transformation is effectively the degree 2 modular equation, any series thus produced would be subsumed under the theory in \cite{CWZ} with the choice $N=2$, and where $\sqrt{\beta}$ could be taken as a singular value. 

If we applied Euler's transformation \eqref{eulert} followed by a quadratic transformation to one of the terms in \eqref{braf1}, however, we arrive at a class of series not explicitly studied in \cite{CWZ} (in the notations used there, the relationship is $\alpha = t_4(1/2+\tau_0)$, $\beta = t_4(\tau_0/2)$). A rational example of such a series is
\begin{equation}
\sum_{n=0}^\infty \binom{2n}{n}^2 P_n\biggl(\frac{2\sqrt2}{3}\biggr)\biggl(\frac{3\sqrt2}{128}\biggr)^n (6n+1) = \frac{2\sqrt{8+6\sqrt2}}{\pi}.
\end{equation}
It is interesting to note that this series has the same $x$ and $z$ as, but is different from, Theorem \ref{thm2leg} with $k=1/\sqrt{2}$. Similarly, with $k=(3-i\sqrt7)/8$ in Theorem  \ref{thm1leg}, we get
\begin{equation}
\sum_{n=0}^\infty \binom{2n}{n}^2 P_n\biggl(\frac{i}{3\sqrt7}\biggr)\biggl(\frac{3i\sqrt7}{256}\biggr)^n (900n^3-564n^2-39n-14) = \frac{384}{\pi},
\end{equation}
which is quite similar to entry (I1) in \cite{CWZ} (first conjectured by Sun); the only difference being that the polynomial term is $16(30n+7)$ in the latter sum.
This phenomenon ultimately stems from the fact that the same (modular) transformations are being used. See also Section \ref{secbraf3} for more discussions. \qed
\end{remark}

\medskip

\subsection{The $s=1/4$ case}

Even though equation \eqref{braf1} holds for $s \in (0,1)$, we see in the last two theorems that transformations need to be applied to the right hand side of \eqref{braf1} before Legendre's relation can be used. Since many such transformations are modular in nature, we are again confined to $s \in \{1/2, 1/3, 1/4,1 /6\}$.
We now consider the $s=1/4$ case in \eqref{braf1}. One strategy here is to transform the right hand side of \eqref{braf1} in terms of $K$; the transformation required is
\[_2F_1\biggl({{\frac14,\frac34}\atop 1};x^2\biggr) = \frac{1}{\sqrt{1+x}} \, _2F_1\biggl({{\frac12,\frac12}\atop 1}; \frac{2x}{1+x}\biggr). \]
The transformed expression is of type \eqref{key} and we solve for $a(z_0)^2 = 1- b(z_0)^2$ in the notation there. Proceeding along the same lines as in the proof of Theorem \ref{thm1leg}, the following theorem can then be established:

\begin{thm} For $k \in (0,1)$, \label{thm3leg} \small
\begin{align} \nonumber
&  \sum_{n=0}^\infty \frac{(\frac14)_n(\frac34)_n}{n!^2} P_n\biggl(\frac{(1+k)(1-4k+7k^2)}{(1-3k)(1+3k^2)}\biggr) \biggl(\frac{(1+k)(1-3k)(1+3k^2)}{(1+3k)^2}\biggr)^n  \\ 
& \quad \times\bigl(C_3 n^3 + C_2 n^2+C_1 n + C_0\bigr) = \frac{3\sqrt2 (1+3k)^{5/2}(1+3k^2)}{(1+k)\pi}, \label{s14c}
\end{align} 
where
\begin{align*}
C_3 & = \frac{16(k-1)^2k^2}{(1+k)^2}(8+15k+9k^2)^2, \\
C_2 & = 48(k-1)k(8-15k+27k^2+27k^3+81k^4), \\
C_1 & = (4-33k+45k^2)(4-17k+17k^2-3k^3+63k^4), \\
C_0 & = 3(1-3k)^4(1+k+2k^2).
\end{align*}
\normalsize
\end{thm}

An example of an identity produced by Theorem \ref{thm3leg} is
\[ \sum_{n=0}^\infty \frac{(\frac14)_n(\frac34)_n}{n!^2} P_n\biggl(\frac{9}{7}\biggr)\biggl(\frac{21}{100}\biggr)^n(216-2385n-108432n^2+80656n^3) = \frac{12600\sqrt{5}}{\pi}. \]
Note that we may also choose $k$ for the right hand side of \eqref{s14c} to be rational.

We can perform a trick here: if the denominator of the argument in $P_n$ is 0 at some $k_0$, and at the same time the geometric term $z_0$ vanishes, when we may take the limit $k \mapsto k_0$ which gets rid of the Legendre polynomial altogether (note that the leading coefficient of $P_n$ is $\binom{2n}{n}2^{-n}$). In \eqref{s14c}, this occurs when $k_0=1/3$. After taking the limit and eliminating the $n^3$ term using a differential equation, we recover the Ramanujan series (of the type \eqref{rama})
\begin{equation} \label{ramaleg1}
\sum_{n=0}^\infty\frac{(\frac14)_n(\frac12)_n(\frac34)_n}{n!^3}\biggl(\frac{32}{81}\biggr)^n (1+7n) =\frac{9}{2\pi}. \end{equation}
The same trick, applied to the series which following from the cubic transformation \eqref{cubicm} mentioned in the $s=1/2$ case, results in
\begin{equation}\label{ramaleg3}
\sum_{n=0}^\infty\frac{(\frac12)_n^3}{n!^3}\biggl(\frac{1}{4}\biggr)^n (1+6n) =\frac{4}{\pi},
\end{equation}
from the choice $p=(\sqrt3-1)/2$; this formula originated from Ramanujan \cite{Ra} and was first proven by Chowla.

\subsection{The $s=1/3$ case}

This case is slightly trickier. An attempt to transform the right hand side of \eqref{braf1} in terms of $K$, as we did for the $s=1/4$ case, results in exceedingly messy computations. Applying low degree modular equations to one of the $_2F_1$'s (as we did in the $s=1/2$ case, for \eqref{braf2} essentially uses the degree 2 modular equation) does not give convergent series. Instead, we resort to a formula in \cite{goursat},
\[ _2F_1\biggl({{\frac13,\frac23}\atop 1};x\biggr) = (1+8x)^{-1/4} \, _2F_1\biggl({{\frac16,\frac56}\atop 1};\frac12-\frac{1-20x-8x^2}{2(1+8x)^{3/2}}\biggr), \]
to transform the right hand side of \eqref{braf1}, then solve for $a(z_0)^2 = 1- b(z_0)^2$ in the notation of \eqref{key}, followed by applying the generalized Legendre relation \eqref{legendregen} with $s=1/3$.
We succeed in obtaining the following theorem, where $\alpha(z_0) = k^3, \, \beta(z_0) = \bigl(\frac{1-k}{1+2k}\big)^3$, and the rate of convergence is $(1+k+k^2)(1-2k+4k^2)^2/(1+2k)^3$.

\begin{thm} \label{thm4leg} For $k \in (0,1)$,
\small
\begin{align} \nonumber
&  \sum_{n=0}^\infty \frac{(\frac13)_n(\frac23)_n}{n!^2} P_n\biggl(\frac{1-4k+6k^2-4k^3+10k^4}{(1-2k-2k^2)(1-2k+4k^2)}\biggr)\biggl(\frac{(1+k+k^2)(1-2k-2k^2)(1-2k+4k^2)}{(1+2k)^3}\biggr)^n \\
& \quad \times \bigl(C_3 n^3 + C_2 n^2+C_1 n + C_0\bigr)= \frac{\sqrt{3} \,(1+2k)^4(1-2k+4k^2)}{\pi}, 
\end{align}
where
\begin{align*}
C_3 & = \frac{9(k-1)^2k^2}{1+k+k^2}(3+4k+2k^2)^2 (3+2k+4k^2)^2, \\
C_2 & = 27(k-1)k(9-18k+10k^2+12k^3+60k^4+160k^5+240k^6+192k^7+64k^8), \\
C_1 & = 9-144k+540k^2-584k^3+314k^4-228k^5-1256k^6-1072k^7+768k^8+2560k^9+1280k^{10}, \\
C_0 & = 2(1-2k-2k^2)^2(1-10k+12k^2-24k^3+16k^4+32k^6).
\end{align*}
\normalsize
\end{thm}

Note that in this case the right hand side contains a surd for rational $k$.  When $k \to (\sqrt3-1)/2$, we get the series
\[  \sum_{n=0}^\infty \frac{(\frac13)_n(\frac12)_n(\frac23)_n}{n!^3}\biggl(\frac{3(7\sqrt3-12)}{2}\biggr)^n \bigl(5-\sqrt3+22n\bigr) = \frac{7+3\sqrt3}{\pi}. \]

\subsection{The $s=1/6$ case}

It is also possible to produce a general series for this case, though the details are formidable and require hours of computer algebra. The derivation is similar to the $s=1/3$ case, and we use Goursat's result \cite{goursat}
\[ _2F_1\biggl({{\frac16,\frac56}\atop 1};\frac12-\frac12 \sqrt{1-\frac{64(1-t)t^3}{(9-8t)^3}}\biggr) = \Bigl(1-\frac{8t}{9}\Bigr)^{\frac14} \, _2F_1\biggl({{\frac13,\frac23}\atop 1};t\biggr), \]
followed by the generalized Legendre relation for $s=1/6$.

The general result for $s=1/6$ is too lengthy to be included here, though in essence its derivation is similar to that of Theorem \eqref{thm1leg} (but with more liberal use of the chain rule). We will only remark on some of its features below. Just to find suitable $x$ (in $P_n$) and $z_0$, we need to solve
\[ \alpha(z_0) = \frac12-\frac12 \sqrt{1-\frac{64(1-t)t^3}{(9-8t)^3}}, \ \beta(z_0) = \frac12-\frac12 \sqrt{1-\frac{64t(1-t)^3}{(1+8t)^3}}, \]
and for both $x$ and $z_0$ to admit rational parametrisations is equivalent to having $1+8t$ and $9-8t$ both as rational squares -- that is, we require a parametrised solution for rational points on the curve $u^2+v^2=10$. Having done so, the resulting series converges for $k \in (1/3, 1)$ where $k$ is the aforementioned parameter; the coefficient of $n$ alone is a degree 24 polynomial in $k$.  Even for $k=1/2$, large integers are involved:
\begin{align*}
& \sum_{n=0}^\infty \frac{(\frac16)_n(\frac56)_n}{n!^2} P_n\biggl(\frac{2437}{2365}\biggr)\biggl(\frac{15136}{296595}\biggr)^n\Bigl(710512440561n^3-118714528800n^2 \\
& \quad -19263658756n-2627089880\Bigr) = \frac{1402894350\sqrt{39}}{\pi}.
\end{align*}
With the limit $k \to (\sqrt5-1)/2$, however, we recover the Ramanujan series
\begin{equation} \label{ramaleg2}
\sum_{n=0}^\infty\frac{(\frac16)_n(\frac12)_n(\frac56)_n}{n!^3}\biggl(\frac{4}{125}\biggr)^n (1+11n) =\frac{5\sqrt{15}}{6\pi}. \end{equation}
In the general series, the $1/\pi$ side is actually the square root of a quartic in $k$, and hence rational points on it may be found by the standard process of converting it to a cubic elliptic curve (namely, $y^2=62208 + 3312x - 144x^2 + x^3$). It follows that there are infinitely many rational solutions. The smallest solution for $k$ (in terms of the size of the denominator) which admits a rational right hand side is $k=6029/8693$, and the resulting series involves integers of over 100 digits.
We include the series in Appendix \ref{verybig} for amusement.

\medskip

\subsection{Rarefied Legendre polynomials}

Factorisations of the type \eqref{key} for generating functions of rarefied Legendre polynomials
\[ \sum_{n=0}^\infty \frac{(\frac12)_n^2}{n!^2} P_{2n}(x)z^{2n}, \ \mathrm{and} \ \sum_{n=0}^\infty \frac{(\frac13)_n(\frac23)_n}{n!^2} P_{3n}(x)z^{3n} \]
are given in \cite{WZ}. Using standard partial differentiation techniques, we may also use Legendre's relation to deduce parameter-dependent rational series for them. The algebra is formidable and we do not present the general forms here; only two examples are given to demonstrate their existence:
\begin{align*}
& \sum_{n=0}^\infty \frac{(\frac12)_n^2}{n!^2}P_{2n}\biggl(\frac{91}{37}\biggr)\biggl(\frac{5}{37}\biggr)^{2n} (3108999168n^3-3255264000n^2-75508700n+24025) \\
& = \frac{896968800}{\pi}, \\
 & \sum_{n=0}^\infty \frac{(\frac13)_n(\frac23)_n}{n!^2}P_{3n}\biggl(\frac{19}{3\sqrt{33}}\biggr) \frac{39887347500n^3-6141658302n^2+172862917n-15262470}{(11\sqrt{33})^n} \\
& = \frac{442203651\sqrt{11}}{2\pi}.
\end{align*}
Under appropriate limits, the series involving $P_{2n}$ again gives \eqref{ramaleg1}, while the one for $P_{3n}$ recovers equation \eqref{ramaleg2}.

\section{Orr-type theorems}

\subsection{A result from Bailey or Brafman}

There are other formulas, notably ones of Orr-type, which satisfy \eqref{key}; an example was given by Bailey \cite[equation (6.3) or (7.2)]{bailey1}:
\begin{equation} 
_4F_3\biggl({{s,s,1-s,1-s}\atop{\frac12,1,1}}; \frac{-x^2}{4(1-x)}\biggr) = {_2F_1}\biggl({{s,1-s}\atop 1};x\biggr){_2F_1}\biggl({{s,1-s}\atop 1};\frac{x}{x-1}\biggr). 
\label{baileys} \end{equation}
Formula \eqref{baileys} is also record in \cite[equation (2.5.32)]{slater} (this reference contains a rich collection of Orr-type theorems). Specializing Bailey's result using $s=1/4$, we have
\begin{equation} \label{baileys1} 
\frac{\pi^2}{4} \, _4F_3\biggl({{\frac14,\frac14,\frac34,\frac34}\atop{\frac12,1,1}}; \frac{-4x^4(1-x^2)^2}{(1-2x^2)^2}\biggr) = K(x)K\biggl(\frac{x}{\sqrt{2x^2-1}}\biggr).
\end{equation}
This formula also follows from setting $x=0$, $s=1/2$ in Brafman's formula \eqref{braf1}. We try different transformations for the right hand side of \eqref{baileys1}, in order to find a suitable $z_0$ for which the two arguments are complementary, so the procedures leading up to \eqref{solved} may be applied. Indeed, after using Euler's transformation \eqref{eulert} to both terms followed by a quadratic transformation, we obtain the equivalent formulation
\[ \frac{\pi^2}{4} \sqrt{\frac{(1+z)(1+z')}{2z'}}\, _4F_3\biggl({{\frac14,\frac14,\frac34,\frac34}\atop{\frac12,1,1}}; \frac{z^4}{4(z^2-1)}\biggr) = K\biggl(\sqrt{\frac{2z}{z+1}}\biggr)K\biggl(\sqrt{\frac{z'-1}{z'+1}}\biggr), \]
where at $z_0=(-1)^{1/6}$ the arguments in the $K$'s are complementary (and correspond to argument $1/4$ in the $_4F_3$). Proceeding as we did for our previous results, Legendre's relation gives
\begin{equation} \label{guic1} \sum_{n=0}^\infty \binom{4n}{2n}^2 \binom{2n}{n} \frac{3+26n+48n^2-96n^3}{2^{12n}} = \frac{2\sqrt{2}}{\pi}. \end{equation}
This time we do not have a more general rational series depending on a parameter, since there is only one free variable $x$ in \eqref{baileys}. For other values of $z_0$, algebraic irrationalities are involved, for instance
\begin{align*}
& \sum_{n=0}^\infty \binom{4n}{2n}^2 \binom{2n}{n} \bigl(-16(62+33\sqrt2)n^3+24(2-\sqrt2)n^2+3(10+\sqrt2)n+3\bigr)\biggl(\frac{2-\sqrt{2}}{2^8}\biggr)^{2n} \\
&= \frac{4\sqrt{2+\sqrt{2}}}{\pi}, \\
& \sum_{n=0}^\infty \binom{4n}{2n}^2 \binom{2n}{n} \bigl(4(151+73\sqrt5)n^3-96(3+\sqrt5)n^2-(25-\sqrt5)n-3\bigr)\biggl(\frac{17\sqrt{5}-38}{2^6}\biggr)^n \\
&= \frac{38+17\sqrt5}{\pi}. 
\end{align*}

We note that it is routine to obtain results \textit{contiguous} to \eqref{baileys1}, i.e.~equations where the left hand side is a $_4F_3$ whose parameters differ from the left hand side of \eqref{baileys1} by some integers. It is known that the corresponding right hand side relates to that of \eqref{baileys1} by a suitable differential operator. Two such contiguous relations give elegant variations of \eqref{guic1}:
\begin{align} \nonumber
\sum_{n=0}^\infty \binom{4n}{2n}^2 \binom{2n}{n} \frac{1-48n^2}{(1-4n)^2 \, 2^{12n}}& = \frac{2\sqrt{2}}{\pi},\\
\sum_{n=0}^\infty \binom{4n}{2n}^2 \binom{2n}{n} \frac{3+32n+48n^2}{(1+2n) \, 2^{12n}} & = \frac{8\sqrt{2}}{\pi}, \label{guic2}
\end{align}
where the second sum has been proven in \cite[table 2]{JG2} using creative telescoping.

In fact, using the same argument $z_0$ as in \eqref{guic1}, we may invoke \eqref{baileys} instead of its specialisation \eqref{baileys1}, and appeal to the generalised Legendre relation. The result, and those contiguous to it, are rather neat and hold for any $s$ that ensures convergence:
\begin{align} \nonumber
& \sum_{n=0}^\infty \frac{(s)_n^2 (1-s)_n^2}{(\frac12)_n(1)_n^3} \frac{s(1-s)+2(1-s+s^2)n+3n^2-6n^3}{(1-2s)^2 \,4^n} \\ \nonumber
= & \sum_{n=0}^\infty \frac{(s)_n^2 (1-s)_n^2}{(\frac32)_n(1)_n^3} \frac{s(1-s)+2n+3n^2}{4^n} \\ 
= & \sum_{n=0}^\infty \frac{(s)_n^2 (-s)_n^2}{(\frac12)_n(1)_n^3} \frac{s^2-3n^2}{s \, 4^n} = \frac{\sin(\pi s)}{\pi}. \label{guic3}
\end{align}
These series generalise \eqref{guic1} and \eqref{guic2}. For rational $s$, the rightmost term in \eqref{guic3} is algebraic; e.g. for $s=1/6$ we get the rational series
\[ \sum_{n=0}^\infty \binom{6n}{4n}\binom{6n}{3n}\binom{4n}{2n} \frac{25-108n^2}{(6n-5)^2 \, 2^{8n} 3^{6n}} = \frac{3}{5\pi}.\]

\subsection{Another result due to Bailey}

We can take \cite[equation (6.1)]{bailey1} (or \cite[(2.5.31)]{slater}), from which we find
\begin{equation} \label{baileys2}
 \pi K\Bigl(\frac{1}{\sqrt2}\Bigr) \, _4F_3\biggl({{\frac14,\frac14,\frac14,\frac34}\atop{\frac12,\frac12,1}}; 16x^2 x'^2 (x'^2-x^2)^2\biggr) = \bigl(K(x)+K'(x)\bigr)K\biggl(\sqrt{\frac12-xx'}\biggr).
\end{equation}
To prepare this identity for Legendre's relation so as to produce even just one rational series, we need to do  more work than we did for \eqref{baileys1}.

We apply the cubic modular equation \eqref{cubicm} to the rightmost term in \eqref{baileys2}. Denoting the $_4F_3$  in \eqref{baileys2} by $G$, we have
\begin{align*}
&   \pi K\Bigl(\frac{1}{\sqrt2}\Bigr) (1+2p) \, G\biggl(\frac{16p^3(1+p)^3(2-p-p^2)(1+2p-4p^3-2p^4)^2}{(1+2p)^4}\biggr) \\
= & \Bigl(K+K'\Bigr)\biggl(\sqrt{\frac12-\frac{\sqrt{p^3(1+p)^3(2-p-p^2)}}{1+2p}}\biggr) \, K\biggl(\sqrt{\frac{p(2+p)^3}{(1+2p)^3}}\biggr).
\end{align*}
At $p = \frac{\sqrt{14}-\sqrt2-2}{4}$ (corresponding to $x^2 = (1-k_7)/2$, where $k_r$ denotes the $r$th singular value of $K$), the arguments in the two $K$'s coincide. We then compute the derivatives up to the 3th order for the above equation. Note that as $G$  satisfies a differential equation of order 4, higher order derivatives are not required; however, since the derivatives also contain the terms $EK, \,E^2$ and $K^2$, we are not a priori guaranteed a solution. After a significant amount of algebra, we amazingly end up with the rational series
\begin{equation}
\sum_{n=0}^\infty \frac{\binom{4n}{2n}\bigl(\frac14\bigr)_n^2}{(2n)!} \frac{5+92n+3120n^2-4032n^3}{2^{8n}}= \frac{8}{K(1/\sqrt2)} = \frac{32\sqrt{\pi}}{\Gamma(\frac14)^2}.
\end{equation}

\subsection{Some related constants} \label{secbraf3}

Using a different set of parameters ($\alpha=\beta=\gamma/2=1/4$ in \cite[equation (6.3)]{bailey1}), we have
\begin{equation} \label{baileys1a}
 \frac{\pi^3}{(2x x')^{\frac12} \, \Gamma(\frac34)^4} \ _3F_2\biggl({{\frac14,\frac14,\frac14}\atop{\frac12,\frac34}}; \frac{(1-2x^2)^4}{16x^2(x^2-1)}\biggr) = \bigl(K(x)+K'(x)\bigr)^2.
\end{equation}
In this case, applying Legendre's relation straightaway does not give anything non-trivial, but if we apply a quadratic transform to the $K(x)$ term first, then for the two arguments in the $K$'s to equal, we need to solve the equation
$\sqrt{1-x^2} = 2\sqrt{x}/(1+x),$ which gives $x=\sqrt2-1$. Subsequently we can use Legendre's relation to obtain
\begin{equation*}
\sum_{n=0}^\infty \frac{(\frac14)_n^4}{(4n)!} \bigl(8(457-325\sqrt2)\bigr)^n \bigl(7+20(11+6\sqrt2)n\bigr) = \frac{28 (82+58\sqrt2)^{\frac14}\, \pi^2}{\Gamma(\frac14)^4}.
\end{equation*}
More series of this type are possible at special values of $t$, which are in fact singular values; c.f.~the equation solved above is precisely the one to solve for the 2nd singular value (because $k_r$ and $k_r'$ are related by the modular equation of degree $r$ and satisfy $k_r^2+k_r'^2=1$). Therefore, to produce a series from \eqref{baileys1a} we do not need Legendre's relation; instead a single differentiation (in the same way Ramanujan series are produced in \cite{Bor}) suffices. For example, using $k_3$ we obtain one series corresponding to $1/\pi$ and another to $1/K(k_r)^2$:
\begin{align} 
\sum_{n=0}^\infty \frac{(\frac14)_n^4}{(4n)!}(-144)^n(1+20n) & = \frac{8\sqrt2\,\pi^2}{\Gamma(\frac14)^4}, \\
\sum_{n=0}^\infty \frac{(\frac14)_n^4}{(4n)!}(-144)^n(5-8n+400n^2) & = \frac{(\frac{2}{\sqrt3}-1)2^{\frac{49}6}\,\pi^5}{\Gamma(\frac13)^6\Gamma(\frac14)^4}.
\end{align}

\section{Concluding remarks}

In equations \eqref{ramaleg1}, \eqref{ramaleg3} and \eqref{ramaleg2}, we witness the ability of Legendre's relation to produce Ramanujan series which have linear (as opposed to cubic) polynomials in $n$. Series of the latter type are often connected with singular values (or more precisely, when $iK'(t)/K(t)$ is a quadratic irrationality), as is further supported by Remark \ref{rmk1leg} and Section \ref{secbraf3}. We take this connection slightly further here. 

We can bypass the need for Brafman's formula completely and produce Ramanujan series of type \eqref{rama} only using Legendre's relation and modular transforms. For instance, take the following version of Clausen's formula,
\begin{equation} \label{3f2trans} _3F_2\left({{\frac12,\frac12,\frac12}\atop{1,1}};4x^2(1-x^2)\right) = \frac{4}{\pi^2(1+x)} K(x) K\Bigl(\frac{2\sqrt{x}}{1+x}\Bigr), \end{equation}
where we have performed a quadratic transformation to get the right hand side. When $x^2+4x/(1+x)^2=1$, $x=\sqrt2-1$, the 2nd singular value. At this $x$, we take a linear combination of the right hand side of \eqref{3f2trans} and  its first derivative (since we know a Ramanujan series exists and involves no higher order derivatives), then apply Legendre's relation \eqref{legendrer}. The result is the series
\[ \sum_{n=0}^\infty \frac{(\frac12)_n^3}{n!^3} \bigl(2(\sqrt2-1)\bigr)^{3n} \bigl(1+(4+\sqrt2)n\bigr) = \frac{3+2\sqrt2}{\pi}, \]
which also follows from \eqref{thm1state} under the limit $k \to \sqrt2-1$. (Applying Legendre's relation to \eqref{3f2trans} and its derivatives when $x$ is not a singular value results in the trivial identity $0=0$, perhaps as expected.)

Applying the quadratic transform twice (i.e. giving the modular equation of degree 4), followed by transforming the $_3F_2$ in \eqref{3f2trans} and using Legendre's relation, we recover Ramanujan's series
\begin{equation}
\sum_{n=0}^\infty\frac{(\frac12)_n^3}{n!^3}\biggl(\frac{-1}{8}\biggr)^n (1+6n) =\frac{2\sqrt2}{\pi}. 
\end{equation}
For our final examples, using the degree 3 modular equation \eqref{cubicm}, we have
\[  _3F_2\left({{\frac12,\frac12,\frac12}\atop{1,1}};\frac{4p^3(1+p)^3(1-p)(2+p)}{(1+2p)^2}\right) = \frac{1}{1+2p} K\biggl(\frac{p^3(2+p)}{1+2p}\biggr)K\biggl(\frac{p(2+p)^3}{(1+2p)^3}\biggr). \]
From this and a similar identity with the $_3F_2$ transformed, we derive the Ramanujan series \eqref{ramaleg3} as well as \eqref{ramaleg2}.
This method seems to be a simple alternative to producing the Ramanujan series \eqref{rama}, since we only need to know the modular equations and Legendre's relation; there is no need to find, say, singular values of the second kind as is required in the approach in \cite{Bor}. 

\subsection{Computational notes}

While all the results presented here are rigorously proven, we outline a method to discover such results  numerically on a computer algebra system. Take the right hand side function in \eqref{key} and compute a linear combination of its derivatives with coefficients $A_i$. Replace the elliptic integrals ($K, K', E, E'$) by $X, X^2, X^4, X^8$ respectively (the indices are powers of 2). Evaluate to several thousand decimal places at the appropriate $z_0$ and collect the coefficients in $X$. Solve for $A_i$ so that Legendre's relation is satisfied (note all the terms such as $KK'$, $E^2$ are separated as different powers of $X$). Finally, identify $A_i$ with an integer relations program like PSLQ.

Many of our (algebraically proven) identities required several hours of computer time due to the complexity of the calculations and the sheer number of steps which needed human direction.  Computational shortcuts, in particular the chain rule, had to be applied manually in order to prevent overflows or out of memory errors. \medskip

\textbf{Acknowledgment:} the author would like to thank Wadim Zudilin for insightful discussions.

\begin{landscape}
\appendix \section{A rational series corresponding to $s=1/6$ in \eqref{braf1}}\label{verybig} \vspace{3cm}
\small
\begin{align*}
& \sum_{n=0}^\infty \frac{(\frac16)_n(\frac56)_n}{n!^2} P_n\biggl(\frac{2711618193169694758695252404104061775156278113}{2258716409636704529221049652293204395071099745}\biggr) \times \\
& \biggl(\frac{166184937571425083357425841157708260870933280014464273}{7435780195982339266650249045977973659251599896998500145}\biggr)^n  \times \\ 
& \bigg\{ -20346216828676992290712717150581898023487858358236131416304525612650289286877
09222379612782511175132608694355 - \\ 
& 926288201652707493107464063135571905251263636429595499527809454558701924211802
5175759095513913027910053062876 n - \\ 
& 1572343886557363398495068846593570880095595404692503176333162218457845
00043816520011553398701765788958763171200 n^2+ \\ 
& 188465262305106044960803571194037001709905551334563776406464312165980638288843
6566607878257622986156264868909056n^3 \bigg\} \\
= \ &   \frac{334603705874692071432125440218193102622139292751902590492288221670640916771026
08071535705583949762471120516075}{2\pi}.
\end{align*}
\normalsize
\end{landscape}

\end{document}